\documentclass[12pt]{amsart}
\usepackage{amscd}
\usepackage{graphicx}
\usepackage{caption}
\def\F{{\mathcal F}}

\def\frk{\frak}               

\def\Phi{{\frk n}}
\def\Phi{{\frk N}}
\def\opn#1#2{\def#1{\operatorname{#2}}} 
\opn\chara{char} \opn\length{\ell} \opn\pd{pd} \opn\rk{rk}
\opn\projdim{proj\,dim} \opn\injdim{inj\,dim} \opn\rank{rank}
\opn\depth{depth} \opn\grade{grade} \opn\height{height}
\opn\embdim{emb\,dim} \opn\codim{codim}

\opn\Tr{Tr} \opn\bigrank{big\,rank}
\opn\superheight{superheight}\opn\lcm{lcm}
\opn\trdeg{tr\,deg}
\opn\reg{reg} \opn\lreg{lreg} \opn\ini{in} \opn\lpd{lpd}
\opn\size{size}
\opn\div{div} \opn\Div{Div} \opn\cl{cl} \opn\Cl{Cl}
\opn\Spec{Spec} \opn\Supp{Supp} \opn\supp{supp} \opn\Sing{Sing}
\opn\Ass{Ass} \opn\Min{Min}
\opn\Ann{Ann} \opn\Rad{Rad} \opn\Soc{Soc}
\opn\Im{Im} \opn\Ker{Ker} \opn\Coker{Coker} \opn\Am{Am}
\opn\Hom{Hom} \opn\Tor{Tor} \opn\Ext{Ext} \opn\End{End}
\opn\Aut{Aut} \opn\id{id}

\opn\nat{nat}
\opn\pff{pf}
\opn\Pf{Pf} \opn\GL{GL} \opn\SL{SL} \opn\mod{mod} \opn\ord{ord}
\opn\Gin{Gin} \opn\Hilb{Hilb}
\opn\aff{aff} \opn\con{conv} \opn\relint{relint} \opn\st{st}
\opn\lk{lk} \opn\cn{cn} \opn\core{core} \opn\vol{vol}
\opn\link{link} \opn\star{star}
\opn\gr{gr}

\def\pot#1#2{#1[\kern-0.28ex[#2]\kern-0.28ex]}

\opn\dirlim{\underrightarrow{\lim}}
\opn\inivlim{\underleftarrow{\lim}}
\def\Implies{\ifmmode\Longrightarrow \else
        \unskip${}\Longrightarrow{}$\ignorespaces\fi}
\def\implies{\ifmmode\Rightarrow \else
        \unskip${}\Rightarrow{}$\ignorespaces\fi}
\def\iff{\ifmmode\Longleftrightarrow \else
        \unskip${}\Longleftrightarrow{}$\ignorespaces\fi}

\let\:=\colon
\newtheorem{Theorem}{Theorem}[section]

\newtheorem{Remark}[Theorem]{Remark}

\newtheorem{Example}[Theorem]{Example}

\newtheorem{Definition}[Theorem]{Definition}

\opn\Syz{Syz} \opn\Im{Im} \opn\Ker{Ker} \opn\Coker{Coker}
\opn\Am{Am} \opn\Hom{Hom} \opn\Tor{Tor} \opn\Ext{Ext} \opn\End{End}
\opn\Aut{Aut} \opn\id{id}

\opn\nat{nat}
\opn\pff{pf}
\opn\Pf{Pf} \opn\GL{GL} \opn\SL{SL} \opn\mod{mod} \opn\ord{ord}
\opn\Gin{Gin}\opn\min{min}
\opn\Hilb{Hilb}\opn\adeg{adeg}\opn\std{std}\opn\ip{infpt}
\opn\Pol{Pol}\opn\sdepth{sdepth}\opn\infpt{infpt}
\opn\depth{depth}\opn\sqdepth{sqdepth}\opn{\Mon}{Mon}

\let\epsilon\varepsilon
\let\phi=\varphi
\let\kappa=\varkappa
\def\qed{\ifhmode\textqed\fi
      \ifmmode\ifinner\quad\qedsymbol\else\dispqed\fi\fi}
\def\textqed{\unskip\nobreak\penalty50
       \hskip2em\hbox{}\nobreak\hfil\qedsymbol
       \parfillskip=0pt \finalhyphendemerits=0}
\def\dispqed{\rlap{\qquad\qedsymbol}}

\opn\dis{dis}
\def\pnt{{\raise0.5mm\hbox{\large\bf.}}}

\opn\Lex{Lex}

\begin{document}
\title
{Characterizations of Line Simplicial Complexes}

\author{Imran Ahmed, Shahid Muhmood}

\address{COMSATS Institute of Information Technology, Lahore, Pakistan}
\email{drimranahmed@ciitlahore.edu.pk}
\address{COMSATS Institute of Information Technology, Lahore, Pakistan}
\email{shahid\_nankana@yahoo.com}

 \maketitle
\begin{abstract} Let $G$ be a finite simple graph. The line graph $L(G)$ represents the adjacencies
between edges of $G$. We define first the line simplicial complex
$\Delta_L(G)$ of $G$ containing Gallai and anti-Gallai simplicial
complexes $\Delta_{\Gamma}(G)$ and $\Delta_{\Gamma'}(G)$
(respectively) as spanning subcomplexes. The study of connectedness
of simplicial complexes is interesting due to various combinatorial
and topological aspects. In Theorem \ref{t1}, we prove that the line
simplicial complex $\Delta_L(G)$ is connected if and only if $G$ is
connected. In Theorem \ref{t2}, we establish the relation between
Euler characteristics of line and Gallai simplicial complexes. In Section 4, we discuss the shellability of line
and anti-Gallai simplicial complexes associated to various classes
of graphs.
 \vskip 0.4 true cm
 \noindent
\end{abstract}

\pagestyle{myheadings} \markboth{\centerline {\scriptsize Ahmed and
Muhmood}}
         {\centerline {\scriptsize Characterizations of Line and Anti-Gallai Simplicial
Complexes}}

\maketitle

\section{Introduction}

Let $\Delta$ be a simplicial complex on the vertex set
$[n]=\{1,\ldots,n\}$ and denote by $\alpha_k$ the number of
$k$-cells of $\Delta$. Then, the Euler characteristic of the
simplicial complex $\Delta$ is given by
$$\chi(\Delta)=\sum\limits_{k=0}^{n-1}(-1)^k\alpha_k.$$

The Euler characteristic is a famous topological and homotopic
invariant to classify surfaces, see \cite{H} and \cite{M}. The
excision is one of the most useful property of Euler characteristic,
given by $\chi(\Delta)=\chi(C)+\chi(\Delta\backslash C)$, for every
closed subset $C\subset \Delta$. The excision property has a dual
form $\chi(\Delta)=\chi(U)+\chi(\Delta\backslash U)$, for every open
subset $U\subset \Delta$.  This property is frequently used under
the guise of the inclusion-exclusion formula.

The shellability of a simplical complex $\Delta$ is a well-studied
combinatorial property that carries strong geometric and algebraic
interpretations, see for example \cite{S}. In many situations,
proving shellability is the most efficient way  of establishing
Cohen-Macaulayness, see for instance \cite{BW}. The algebraic
criterion for the shellability of a simplicial complex has been
firstly introduced by A. Dress \cite{D}. In \cite{ER}, Eagon and
Reiner gave algebraic criterion of the pure shellability of a dual
simplicial complex $\check{\Delta}$ in the context of the
Stanley-Reisner ideal theory.

Recently, in \cite{AKNS}, Anwar, Kosar and Nazir gave a translation
of the shellability of a simplicial complex $\Delta$ on the monomial
generators of the facet ideal $I_{\mathcal{F}}(\Delta)$. Their
algebraic translation provided an useful class of ideals known as
ideals with Linear residuals.

Let $G$ be a finite simple graph. The line graph $L(G)$ of $G$ is a
graph having edges of $G$ as its vertices and two distinct edges of
$G$ are adjacent in $L(G)$ if they are adjacent in $G$. It was
firstly introduced by Harary and Norman in \cite{HN}.

Both the Gallai and anti-Gallai graphs $\Gamma(G)$ and $\Gamma'(G)$
of a graph $G$ have the edges of $G$ as their vertices. Two edges of
$G$ are adjacent in the Gallai graph $\Gamma(G)$ if they are
incident but do not span a triangle in $G$; they are adjacent in the
anti-Gallai graph $\Gamma'(G)$ if they span a triangle in $G$, see
\cite{TG} and \cite{VBL}. The Gallai and anti-Gallai graphs are
spanning subgraphs of the line graph $L(G)$. The anti-Gallai graph
$\Gamma'(G)$ is the complement of $\Gamma(G)$ in $L(G)$.

We define first the line simplicial complex $\Delta_L(G)$ of $G$
containing Gallai and anti-Gallai simplicial complexes
$\Delta_{\Gamma}(G)$ and $\Delta_{\Gamma'}(G)$ (respectively) as
spanning subcomplexes. The study of connectedness of simplicial
complexes is interesting due to various combinatorial and
topological aspects, see \cite{BP} and \cite{MABY}. In Theorem
\ref{t1}, we prove that the line simplicial complex $\Delta_L(G)$ is
connected if and only if $G$ is connected. In Theorem \ref{t2}, we
establish the relation between Euler characteristics of line and
Gallai simplicial complexes.

In Section 4, we discuss the shellability of line
and anti-Gallai simplicial complexes associated to various classes
of graphs.

\section{Preliminaries}

A simplicial complex $\Delta$ on the vertex set $[{n}]=\{1,\ldots,
n\}$ is a subset of $2^{[n]}$ with the property that
if $F\in \Delta$ then every subset of $F$ will belong to $\Delta$.
The members of $\Delta$ are called faces and the maximal faces under inclusion are called facets. If $\mathcal{F}(\Delta)=\{F_1,\ldots ,F_h\}$ is the set of all facets of $\Delta$, then $\Delta=<F_1,\ldots ,F_h>$. A subcomplex of the simplicial complex $\Delta$ is a simplicial complex whose facet set is a subset of $\mathcal{F}(\Delta)$. The dimension of a face $F\in\Delta$ is given by $\dim F=\mid F\mid-1$, where $\mid F\mid$ is the number of vertices of $F$. The dimension of a simplical complex $\Delta$ is defined by $\dim\Delta=max\{\dim F\ |\
F\in\Delta\}$. A simplicial complex $\Delta$ is
said to be pure if it has all facets of the same dimension.

A simplicial complex $\Delta$ is said to be connected if for any two
facets $F$ and $\tilde{F}$ of $\Delta$, there exists a sequence of
facets $F=F_0,\ldots,F_q=\tilde{F}$ such that $F_i\cap F_{i+1}\neq
\emptyset$ for any $i=0,\ldots,q-1$. A disconnected simplicial complex
is a complex which is not connected. That is, the vertex set $[n]$ of
$\Delta$ can be written as disjoint union of two non-empty subsets
$V_1$ and $V_2$ of $[n]$ such that no face of $\Delta$ has vertices in
both $V_1$ and $V_2$, see \cite{BP} and \cite{MABY}.

We define now the line graph $L(G)$, which provides the main
streamline of this work, see \cite{HN}.
\begin{Definition}
\rm{Let $G$ be a finite simple graph. The graph $L(G)$ is said to be
line graph of $G$ if each vertex of $L(G)$ represents an edge of $G$
and two vertices of $L(G)$ are adjacent if and only if their
corresponding edges are incident in $G$.}
\end{Definition}
\begin{Example} {\rm The graph $G$ and its
Line graph $L(G)$ are given in Figure ~\ref{fig:f1}.
\begin{figure}[htbp]
        \centerline{\includegraphics[width=14.0cm]{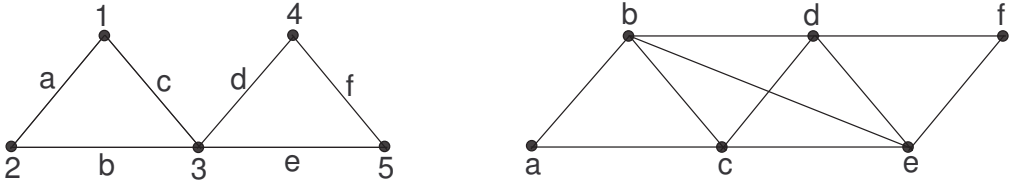}}
        \caption {Graph $G$ and its Line Graph $L(G)$}
         \label{fig:f1}
        \end{figure}}
\end{Example}
\section{Topological Characterizations of Line Simplicial Complexes}

The following definition plays a key role in the structural study of line graph $L(G)$.
\begin{Definition} {\rm Let $G$ be a  finite simple graph on the vertex set
$V(G)=[n]$. Let $E(G)=\{e_{i,j}=\{i,j\}|i,j\in V(G)\}$ be the edge set of $G$. We
define the set of line indices $\Upsilon(G)$ associated to the graph $G$ as the
collection of subsets of $V(G)$ such that if $e_{i,j}$ and $e_{j,k}$
are adjacent in $L(G)$, then $F_{i,j,k}=\{i,j,k\}\in \Upsilon(G)$ or
if $e_{i,j}$ is an isolated vertex in $L(G)$ then
$F_{i,j}=\{i,j\}\in \Upsilon(G)$.

The line index is said to be a Gallai index if the incident edges $e_{i,j}$ and $e_{j,k}$ of $G$ do not span a triangle in $G$. We denote the set of Gallai indices by $\Omega_{\Gamma}(G)$, see \cite{AM} and \cite{AKNS}. The line index is said to be an anti-Gallai index if the incident edges $e_{i,j}$ and $e_{j,k}$ of $G$ lie on a triangle in $G$. We denote the set of anti-Gallai indices by $\Omega_{\Gamma'}(G)$. The set of line indices $\Upsilon(G)$ contains $\Omega_{\Gamma}(G)$ and $\Omega_{\Gamma'}(G)$ as spanning subsets.}
\end{Definition}

\begin{Definition} {\rm A line simplicial complex $\Delta_L(G)$ of $G$ is
a simplicial complex on the vertex set $V(G)$ such that
$$\Delta_L(G)=<F \ | \ F\in \Upsilon(G)>,$$ where $\Upsilon(G)$ is
the set of line indices of $G$.}
\end{Definition}

 Similarly, the Gallai and anti-Gallai simplicial complexes $\Delta_{\Gamma}(G)$ and $\Delta_{\Gamma'}(G)$
 are generated by Gallai and anti-Gallai indices, respectively. We refer \cite{AM} and \cite{AKNS} for Gallai
 simplicial complex. The line simplicial complex $\Delta_L(G)$ contains Gallai and anti-Gallai simplicial
 complexes as spanning subcomplexes. The anti-Gallai simplicial complex $\Delta_{\Gamma'}(G)$ is complement of
 $\Delta_{\Gamma}(G)$ in $\Delta_L(G)$.

We prove first necessary and sufficient condition for
connectedness of the line simplicial complex $\Delta_L(G)$.

\begin{Theorem}\label{t1}
Let $G$ be a finite simple graph on the vertex set $[n]$. Then, $G$ is connected if
and only if the line simplicial complex $\Delta_L(G)$ is connected.
\end{Theorem}
\begin{proof}
Let $G$ be a finite simple graph on the vertex set $[n]$. For $n\leq
3$, the result is trivial. Therefore, we take $n\geq 4$.

We establish first direct implication. On contrary, we assume
that the line simplicial complex $\Delta_L(G)$ is not connected. By
definition, there exists two non-empty subsets $V_1$ and $V_2$ of
$[n]$ such that $[n]=V_1\cup V_2$ and $V_1\cap V_2=\emptyset$ with
the property that any facet of $\Delta_L(G)$ either has vertices
from $V_1$ or $V_2$. Since $G$ is connected graph on the vertex set $[n]$ with
$n\geq 4$, therefore the line simplicial complex
$\Delta_L(G)$ is pure of dimension $2$. So, there
are facets, say $F_{i,j,k}, F_{l,m,p}\in\Delta_L(G)$ for every
$i\neq j\neq k\in V_1\subset [n]$ and for every $l\neq m\neq p\in
V_2\subset [n]$ such that $e_{i,j},e_{j,k}$ and
$e_{l,m},e_{m,p}$ are adjacent vertices of the line graph $L(G)$. Therefore,
the edges $e_{i,j},e_{j,k}$ and $e_{l,m},e_{m,p}$ are incident in
$G$ for every $i,j,k\in V_1$ and for every $l,m,p\in V_2$  with $[n]=V_1\cup V_2$ such that $V_1\cap V_2=\emptyset$, a
contradiction.

Now, we prove converse implication. On contrary,
we assume that the graph $G$ is not connected. That is, there exist
two vertices, say $r,s\in G$ such that no path in $G$ has $r$ and
$s$ as end points. It implies that there is no face of $\Delta_L(G)$
containing both vertices $r$ and $s$ i.e. $\Delta_L(G)$ is not
connected, a contradiction. Hence the result.
\end{proof}

We establish now the relation between Euler characteristics of line and Gallai simplicial complexes.

\begin{Theorem}\label{t2}
Let $\Delta_L(G)$ and $\Delta_{\Gamma}(G)$ be line and Gallai
simplicial complexes of a finite simple graph $G$. Then,
the Euler characteristic of the line simplicial complex $\Delta_L(G)$ is given by
$$\chi(\Delta_L(G))=\chi(\Delta_{\Gamma}(G))+\mid \Omega_{\Gamma'}(G)\mid,$$
where $\mid \Omega_{\Gamma'}(G)\mid$ is the number of anti-Gallai
indices associated to $G$.
\end{Theorem}

\begin{proof}
Let $G$ be a finite simple graph consisting of $t$ connected
components $G_1,\ldots,G_t$. Then, $G=\cup_{k=1}^t\,\, G_k$ such
that $G_i\cap G_j=\emptyset$ for every $i,j=1,\ldots,t$ with $i\neq
j$. By Theorem \ref{t1}, the line simplicial complex
$\Delta_L(G)$ also consists of $t$ connected components
$\Delta_L(G_1),\ldots,\Delta_L(G_t)$. Therefore, the line simplicial complex can be expressed as
$\Delta_L(G)=\cup_{k=1}^t\,\,\Delta_L(G_k)$ such that
$\Delta_L(G_i)\cap\Delta_L(G_j)=\emptyset$ for all
$i,j=1,\ldots,t$ with $i\neq j$.

By the excision property, the Euler characteristic of the line
simplicial complex $\Delta_L(G)$ is given by
$$\chi(\Delta_L(G))=\sum\limits_{k=1}^t\chi(\Delta_L(G_k)).$$
Let $\Delta_{\Gamma}(G_k)$ and $\Delta_{\Gamma'}(G_k)$ be Gallai and anti-Gallai simplicial complexes associated to each connected component $G_k$ of $G$ for $k=1,\ldots,t$. Then, each connected component $\Delta_L(G_k)$ of line simplicial complex contains $\Delta_{\Gamma}(G_k)$ and $\Delta_{\Gamma'}(G_k)$ as spanning subcomplexes for $k=1,\ldots,t$. Therefore, by definition, the Euler characteristic of each connected component $\Delta_L(G_k)$  of line simplicial complex is given by
$$\chi(\Delta_L(G_k))=\chi(\Delta_{\Gamma}(G_k))+\mid \Omega_{\Gamma'}(G_k)\mid,$$
where $\Omega_{\Gamma'}(G_k)$ is the set of anti-Gallai indices associated to $G_k$  for $k=1,\ldots,t$.
Consequently,
$$\chi(\Delta_L(G))=\sum\limits_{k=1}^t\chi(\Delta_{\Gamma}(G_k))+\sum\limits_{k=1}^t\mid \Omega_{\Gamma'}(G_k)\mid=\chi(\Delta_{\Gamma}(G))+\mid \Omega_{\Gamma'}(G)\mid$$
due to excision property. Hence proved.
\end{proof}

\begin{Example}
\rm{Consider the graph $G$ given in Figure ~\ref{fig:f2}.
\begin{figure}[htbp]
        \centerline{\includegraphics[width=4.0cm]{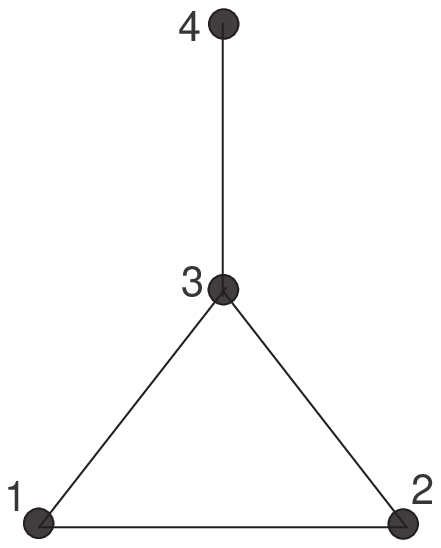}}
        \caption {Graph $G$}
         \label{fig:f2}
        \end{figure}}
\end{Example}
Then $\Delta_L(G)=<\{1,2,3\},\{1,3,4\},\{2,3,4\}>$ and
$\chi(\Delta_L(G))=f_0-f_1+f_2=4-6+3=1$, where $f_i$ is the
$i$-dimensional faces of $\Delta_L(G)$.\\ Also,
$\Delta_{\Gamma}(G)=<\{1,2\},\{1,3,4\},\{2,3,4\}>$ and
$\chi(\Delta_{\Gamma}(G))=g_0-g_1+g_2=4-6+2=0$, where $g_i$ is the
$i$-dimensional faces of $\Delta_{\Gamma}(G)$.\\
And $|\Omega_{\Gamma'}(G)|=|\{1,2,3\}|=1$. Therefore,
$$\chi(\Delta_L(G))=\chi(\Delta_{\Gamma}(G))+\mid \Omega_{\Gamma'}(G)\mid,$$
where $\mid \Omega_{\Gamma'}(G)\mid$ is the number of anti-Gallai
indices associated to $G$.
\begin{Example}\label{ex1}
\rm{Let $W_{n+1}$ be the wheel graph on $n+1$ vertices having edge
set
$E(W_{n+1})=\{e_{1,2},\ldots,e_{n,1},e_{1,n+1},\ldots,e_{n,n+1}\},$
as shown in the Figure ~\ref{fig:f3}.
\begin{figure}[htbp]
        \centerline{\includegraphics[width=7.0cm]{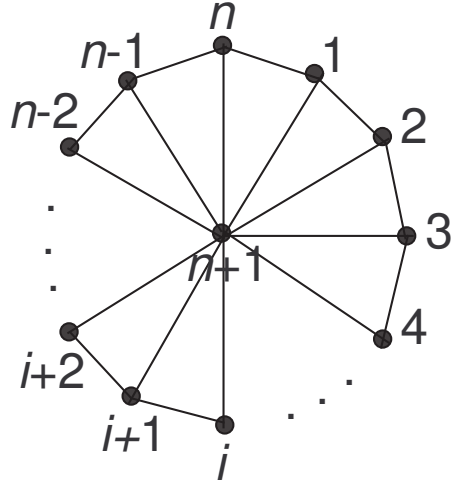}}
        \caption {Wheel Graph $W_{n+1}$}
         \label{fig:f3}
        \end{figure}}
\end{Example}
Then, the line indices of the wheel graph are given by\\
$\Upsilon(W_{n+1})=\{F_{1,2,3},\ldots,F_{n,1,2},F_{1,2,n+1},F_{1,3,n+1},\ldots,F_{1,n,n+1},
F_{2,3,n+1},
\ldots,\\F_{2,n,n+1},\ldots,F_{n-1,n,n+1}\}$. We compute first the
Euler characteristic of the line simplicial complex
$\Delta_L(W_{n+1})$ for $n\geq 4$. There are $(n+1)$ 0-dimensional
faces or vertices in $\Delta_L(W_{n+1})$ i.e. $\alpha_0=n+1$. Now,
the number of $1$-dimensional faces of $\Delta_L(W_{n+1})$ is given
by $\alpha_1=|\{j,k\}|={n+1 \choose 2}=\frac{n(n+1)}{2}$, where
$1\leq j\neq k\leq n+1$. Next, we compute the number of
$2$-dimensional faces of the form  $\{j,k,l\}\in \Delta_L(W_{n+1})$
with $1\leq j, k\leq n$ and $1\leq l\leq n+1$ such that $j\neq k\neq
l$.
\begin{enumerate}
  \item $|\{j,j+1,j+2\}|=n-2$, where $1\leq j\leq n-2$;
  \item $|\{n-1,n,1\}|=1$;
  \item $|\{n,1,2\}|=1$;
  \item $|\{j,k,n+1\}|={n \choose 2}=\frac{n(n-1)}{2}$, where $1\leq j\neq
  k\leq n$.
\end{enumerate}
Adding from $(1)$ to $(4)$, we get $\alpha_2=n-2+1+1+\frac{n(n-1)}{2}=\frac{n(n+1)}{2}$.\\
Thus, we obtain
$$\chi(\Delta_L(W_{n+1}))=\alpha_0-\alpha_1+\alpha_2=(n+1)-\frac{n(n+1)}{2}+\frac{n(n+1)}{2}=n+1,$$
where $n\geq 4$. By definition, the anti-Gallai indices of the wheel graph $W_{n+1}$ are given by
$$\Omega_{\Gamma'}(W_{n+1})=\{F_{1,2,n+1},F_{2,3,n+1},\ldots,F_{n,1,n+1}\}.$$
It implies that $\mid\Omega_{\Gamma'}(W_{n+1})\mid=n$. Therefore, by Theorem \ref{t2},
$$\chi(\Delta_{\Gamma}(W_{n+1}))=\chi(\Delta_L(W_{n+1}))-\mid\Omega_{\Gamma'}(W_{n+1})\mid=n+1-n=1$$
with $n\geq 4$.
\begin{Example}\label{ex2}
\rm{Let $F_n$ be Friendship graph on $2n+1$ vertices with edge set
$$E(F_{n})=\{e_{1,2},e_{3,4},\ldots,e_{2n-1,2n},e_{1,2n+1},e_{2,2n+1},\ldots,e_{2n,2n+1}\},$$
as shown in the Figure ~\ref{fig:f4}.
\begin{figure}[htbp]
        \centerline{\includegraphics[width=8.0cm]{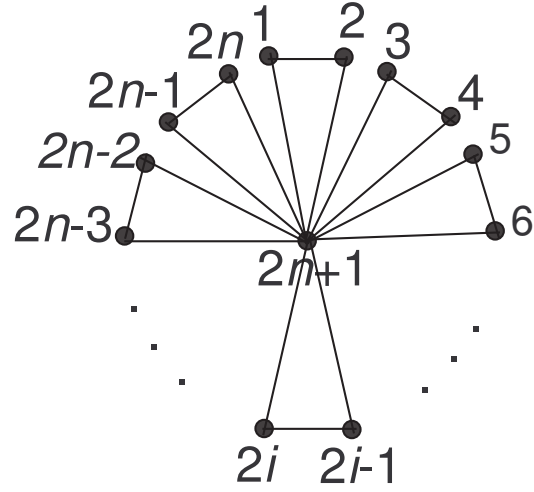}}
        \caption {Friendship Graph $F_{n}$}
         \label{fig:f4}
        \end{figure}}
\end{Example}
Then, the Gallai indices of $F_{n}$ are given by\\
$\Omega_{\Gamma}(F_{n})=\{F_{1,2},\ldots,F_{2n-1,2n},F_{1,3,2n+1},\ldots,F_{1,2n,2n+1},
F_{2,3,2n+1},\ldots,
F_{2,2n,2n+1},\\F_{3,5,2n+1},\ldots,F_{2n-2,2n,2n+1}\}$, see \cite{AKNS}. We compute first the Euler characteristic of the Gallai simplicial complex $\Delta_{\Gamma}(F_{n})$ for $n\geq 2$.
Since, the friendship graph $F_{n}$ has $(2n+1)$ vertices, therefore the number of $0$-dimensional faces in $\Delta_{\Gamma}(F_{n})$ is
$\alpha_0=2n+1$. Moreover, the number of $1$-dimensional faces of $\Delta_{\Gamma}(F_{n})$ is given by $\alpha_1=|\{i,j\}|={2n+1 \choose 2}=n(2n+1)$, where $1\leq i\neq j\leq
2n+1$. Now, we compute the number of $2$-dimensional faces of the form $\{i,j,k\}\in \Delta_{\Gamma}(F_{n})$ with $1\leq i\leq 2n-2$, $3\leq j\leq 2n$ and $k=2n+1$ such that $i\neq j$.\\
(1)\,\,\,\,\, $|\{i,j,2n+1\}|=2(2n-2)$ with $i\in \{1,2\}$ and $3\leq j\leq
2n$;\\
(2)\,\,\,\,\, $|\{i,j,2n+1\}|=2(2n-4)$ with $i\in \{3,4\}$ and $5\leq j\leq 2n$;\\
  $\vdots$\\
(n-2)\, $|\{i,j,2n+1\}|=2(4)$ with $i\in \{2n-5,2n-4\}$ and
$2n-3\leq j\leq 2n$;\\
(n-1)\, $|\{i,j,2n+1\}|=2(2)$ with $i\in \{2n-3,2n-2\}$ and
$2n-1\leq j\leq 2n$.\\
Adding from $(1)$ to $(n-1)$, we obtain\\ $\alpha_2=2(2n-2)+2(2n-4)+\cdots+2(4)+2(2)=2n(n-1)$.\\
Therefore, we compute
$$\chi(\Delta_{\Gamma}(F_{n}))=\alpha_0-\alpha_1+\alpha_2=(2n+1)-n(2n+1)+2n(n-1)=1-n,$$
where $n\geq 2$. Note that $\mid\Omega_{\Gamma'}(F_n)\mid=n$. Hence,
by Theorem \ref{t2}, $\chi(\Delta_L(F_{n}))=1$ with $n\geq 2$.
%
\begin{Remark}
\rm{Let $G$ be a finite simple graph. If there is no triangle in $G$, then there will be no anti-Gallai index in $\Upsilon(G)$ i.e. $\Omega_{\Gamma'}(G)=\emptyset$ and the line and Gallai simplicial complexes of $G$ are coincident.

It can be easily seen that the line simplicial complexes
associated to friendship graph $F_{n}$ and star graph $S_{2n}$ are
the same.}
\end{Remark}

\section{Shellability of Line and Anti-Gallai Simplicial Complexes}
We introduce first a few notions.
\begin{Definition}
\rm{A simplicial complex $\Delta$ over $[n]$ is $\textbf{shellable}$
if its facets can be ordered $F_1,F_2,\ldots,F_s$ such that, for all
$2\leq i\leq s$ the subcomplex
$$\hat{\Delta}_{<F_i>}=<F_1,F_2,\ldots,F_{i-1}>\cap <F_i>$$ is pure
of dimension $dim(F_i)-1$.}
\end{Definition}

\begin{Definition}
\rm{Let $I\subset S=k[x_1,\ldots,x_n]$ be a monomial ideal. We say
that $I$ has $\textbf{Linear Residuals}$ if there exist an ordered
minimal monomial system of generators $\{m_1,m_2,\ldots,m_r\}$ of
$I$ such that Res$(I_i)$ is minimally generated by linear monomials
for $1< i \leq r$, where Res$(I_i)=\{u_1,u_2,\ldots,u_{i-1}\}$ such
that $u_k=\frac{m_i}{gcd(m_k,m_i)}$ for all $1\leq k\leq i-1$.}
\end{Definition}
The following result provides effective necessary and sufficient
condition for the shellability of a simplicial complex, see
\cite{AKNS}.
\begin{Theorem} \label{t3} \cite{AKNS} \rm{Let $\Delta$ be a simplicial complex of
dimension $d$ over $[n]$. Then $\Delta$ will be shellable if and
only if $I_\mathcal{F}(\Delta)$ has linear residuals.}
\end{Theorem}

\begin{Theorem}\label{t6} \rm{The line simplicial complex $\Delta_L(F_n)$ associated to friendship
graph $F_n$ is shellable.}
\end{Theorem}
\begin{proof} By Theorem \ref{t3}, it is sufficient to show that
$I_{\mathcal{F}}(\Delta_L(F_n))$ have linear residuals. As the line
indices of the friendship graph $F_n$ on $2n+1$ vertices are given
by
$$\Upsilon_{L}(F_{n})=\{F_{1,2,2n+1},\ldots,F_{1,2n,2n+1},F_{2,3,2n+1},\ldots,F_{2,2n,2n+1},
\ldots, F_{2n-1,2n,2n+1}\},$$ as shown in Figure ~\ref{fig:f4}. Then, the ordered minimal monomial system of generators are given by\\
$I_{\mathcal{F}}(\Delta_L(F_{n}))=(m_{F_{1,2,2n+1}},\ldots,m_{F_{1,2n,2n+1}},m_{F_{2,3,2n+1}},\ldots,m_{F_{2,2n,2n+1}},\\
\ldots,m_{F_{2n-1,2n,2n+1}})$, where $m_{F_{i,j,2n+1}}$ are the
monomial $x_ix_jx_{2n+1}$. One can easily see that
Res$(I_{m_{F_{1,j,2n+1}}})$ is minimally generated by
$$\frac{m_{F_{1,j,2n+1}}}{gcd(m_{F_{1,j-1,2n+1}},m_{F_{1,j,2n+1}})}=x_j \ ,\ \ \ \ \ \  3\leq j\leq 2n$$
Moreover, Res$(I_{m_{F_{i,j,2n+1}}})$ with $1<i<j\leq 2n$ is
minimally generated by the linear monomials $x_j$ and $x_i$ due to
$$\frac{m_{F_{i,j,2n+1}}}{gcd(m_{F_{1,i,2n+1}},m_{F_{i,j,2n+1}})}=x_j$$
and
$$\frac{m_{F_{i,j,2n+1}}}{gcd(m_{F_{1,j,2n+1}},m_{F_{i,j,2n+1}})}=x_i.$$

\end{proof}

\begin{Theorem}\label{t7} \rm{The line simplicial complex $\Delta_L(W_{n+1})$ associated to
wheel graph $W_{n+1}$ is shellable.}
\end{Theorem}
\begin{proof}   The ordered minimal monomial system of generators are given by\\
$I_{\mathcal{F}}(\Delta_L(W_{n+1}))=(m_{F_{1,2,n+1}},\ldots,m_{F_{1,n,n+1}},m_{F_{2,3,n+1}},\ldots,m_{F_{2,n,n+1}},
\ldots,m_{F_{n-1,n,n+1}},\\ m_{F_{1,2,3}},\ldots,m_{F_{n,1,2}})$,
where $m_{F_{i,j,k}}$ are the monomials $x_ix_jx_k$ associated to the facets $F_{i,j,k}$, see
Figure ~\ref{fig:f3}. We establish the result
into the following steps.\\
{\bf{Step.I.}} For the monomials
$m_{F_{1,2,n+1}},\ldots,m_{F_{1,n,n+1}}$, one can easily see that
Res$(I_{m_{F_{1,j,n+1}}})$ is minimally generated by
$$\frac{m_{F_{1,j,n+1}}}{gcd(m_{F_{1,j-1,n+1}},m_{F_{1,j,n+1}})}=x_j \ ,\ \ \ \ \ \  3\leq j\leq n.$$
{\bf{Step.II.}} For the monomials
$m_{F_{2,3,n+1}},\ldots,m_{F_{2,n,n+1}}, \ldots,m_{F_{n-1,n,n+1}}$,
Res$(I_{m_{F_{i,j,n+1}}})$ with $1<i<j\leq n$ is minimally
generated by the linear monomials $x_i$ and $x_j$ due to
$$\frac{m_{F_{i,j,n+1}}}{gcd(m_{F_{1,i,n+1}},m_{F_{i,j,n+1}})}=x_j$$
and
$$\frac{m_{F_{i,j,n+1}}}{gcd(m_{F_{1,j,n+1}},m_{F_{i,j,n+1}})}=x_i.$$
{\bf{Step.III.}} For the monomials
$m_{F_{1,2,3}},\ldots,m_{F_{n-2,n-1,n}}$, Res$(I_{m_{F_{i,j,k}}})$
with $1\leq i<j<k\leq n$ is minimally generated by the linear monomials
$x_i,x_j$ and $x_k$ due to
$$\frac{m_{F_{i,j,k}}}{gcd(m_{F_{i,j,n+1}},m_{F_{i,j,k}})}=x_k,$$
$$\frac{m_{F_{i,j,k}}}{gcd(m_{F_{j,k,n+1}},m_{F_{i,j,k}})}=x_i$$ and
$$\frac{m_{F_{i,j,k}}}{gcd(m_{F_{i,k,n+1}},m_{F_{i,j,k}})}=x_j.$$
{\bf{Step.IV.}} Finally, one can easily see that the residuals
\begin{center}
Res$(I_{m_{F_{n-1,n,1}}})=\{x_1, x_{n-1}, x_n\}$ and
Res$(I_{m_{F_{n,1,2}}})=\{x_1, x_2, x_n\}$ 
\end{center}
are minimally generated by linear monomials.
\end{proof}

\begin{Example} \rm{
Consider the line simplicial complex
$$\Delta_L(C_5)=\langle
\{1,2,3\},\{2,3,4\},\{3,4,5\},\{4,5,1\},\{5,1,2\}\rangle$$
associated to the cycle $C_5=(12345)$ on $5$ vertices. Then
$$I_{\mathcal{F}}(\Delta_L(C_5))=(m_{F_{1,2,3}},m_{F_{2,3,4}},m_{F_{3,4,5}},m_{F_{4,5,1}},m_{F_{5,1,2}}).$$
On contrary, we assume that $\Delta_L(C_5)$ is shellable. Therefore, by
Theorem \ref{t3}, $I_{\mathcal{F}}(\Delta_L(C_5))$ admits linear
residuals. Without loss of generality, we may assume that
$\mathbf{m_1}=m_{F_{1,2,3}}$ and $\mathbf{m_2}=m_{F_{2,3,4}}$. If we take
$\mathbf{m_3}=m_{F_{4,5,1}}$, then we have Res$(I_{\mathbf{m_3}})=\{x_4x_5,x_1x_5\}$
not generated by linear monomials. Therefore, $\mathbf{m_3}\neq
m_{F_{4,5,1}}$. So, we have either $\mathbf{m_3}=m_{F_{5,1,2}}$ or
$\mathbf{m_3}=m_{F_{3,4,5}}$. If $\mathbf{m_3}=m_{F_{5,1,2}}$, we have either
$\mathbf{m_4}=m_{F_{4,5,1}}$ or $\mathbf{m_4}=m_{F_{3,4,5}}$. Then, the residuals\\
$Res(I_{\mathbf{m_4}})=Res(I_{m_{F_{4,5,1}}})=\{x_1x_5,x_4\}$ and $Res(I_{\mathbf{m_4}})=Res(I_{m_{F_{3,4,5}}})=\{x_5,x_3x_4\}$\\
are not
minimally generated by linear monomials.\\  
If $\mathbf{m_3}=m_{F_{3,4,5}}$, then either
$\mathbf{m_4}=m_{F_{4,5,1}}$ or $\mathbf{m_4}=m_{F_{5,1,2}}$. Then, the residuals\\
$Res(I_{\mathbf{m_4}})=Res(I_{m_{F_{4,5,1}}})=\{x_1,x_4x_5\}$ and $Res(I_{\mathbf{m_4}})=Res(I_{m_{F_{5,1,2}}})=\{x_5,x_1x_2\}$\\
are not
minimally generated by linear monomials, which is a contradiction.}
\end{Example}

\begin{Theorem}\label{t5} The anti-Gallai simplicial complex
$\Delta_{\Gamma'}(W_{n+1})$ associated to wheel graph $W_{n+1}$ is
shellable.
\end{Theorem}

\begin{proof}
The
anti-Gallai indices of wheel graph $W_{n+1}$ are given by
$$\Omega_{\Gamma'}(W_{n+1})=\{F_{1,2,n+1},F_{2,3,n+1},\ldots,F_{n-1,n,n+1},F_{n,1,n+1}\},$$
as shown in Figure ~\ref{fig:f7}.

\begin{figure}[htbp]
        \centerline{\includegraphics[width=7.0cm]{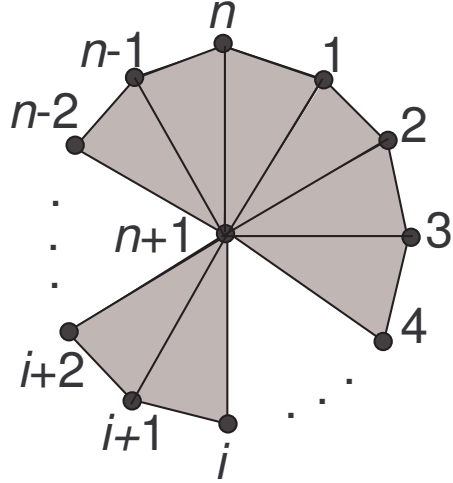}}
        \caption {Anti-Gallai simplicial complex $\Delta_{\Gamma'}(W_{n+1})$}
         \label{fig:f7}
        \end{figure}
Then, the facet ideal $I_\mathcal{F}(\Delta_{\Gamma'}(W_{n+1}))$ is
given by
$$I_\mathcal{F}(\Delta_{\Gamma'}(W_{n+1}))=(x_{F_{1,2,n+1}},x_{F_{2,3,n+1}},\ldots,x_{F_{n-1,n,n+1}},x_{F_{n,1,n+1}}),$$
where $x_{F_{i,j,n+1}}$ are the monomials $x_i x_j x_{n+1}$. It can be
easily seen that Res$(I_{x_{F_{i,i+1,n+1}}})$ is minimally generated
by
$$\frac{m_{F_{i,i+1,n+1}}}{gcd(m_{F_{i-1,i,n+1}},m_{F_{i,i+1,n+1}})}=x_{i+1} \ ,\ \ \ \ 2\leq i\leq n-1. $$
Moreover,
Res$(I_{x_{F_{n,1,n+1}}})=\{x_1,x_n\}$. Thus, the
anti-Gallai simplicial complex $\Delta_{\Gamma'}(W_{n+1})$ is
shellable.
\end{proof}

In the following examples, we see that the anti-Gallai simplicial
complex $\Delta_{\Gamma'(G)}$ are not shellable.
\begin{Example}
\rm{Let $G=Y_{3,n}$ be prism graph  having $3n$ vertices and
$3(2n-1)$ edges, as shown in \cite{AM}. The anti-Gallai indices
associated to the prism graph $Y_{3,n}$ are given by
$$\Omega_{\Gamma'}(Y_{3,n})=\{F_{1,2,3},F_{4,5,6},\ldots,F_{3n-2,3n-1,3n}\}.$$
The anti-Gallai simplicial complex $\Delta_{\Gamma'}(Y_{3,n})$
consisting of $n$ disjoint facets is pure of dimension $2$, as shown
in Figure ~\ref{fig:f8}.
\begin{figure}[htbp]
        \centerline{\includegraphics[width=10.0cm]{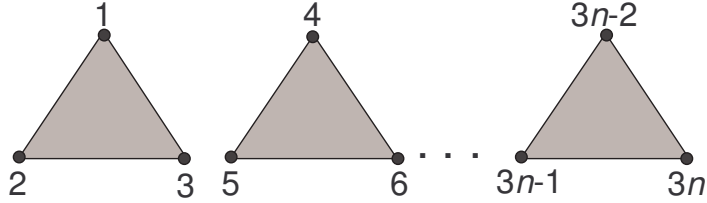}}
        \caption {Anti-Gallai simplicial complex $\Delta_{\Gamma'}(Y_{3,n})$}
         \label{fig:f8}
        \end{figure}}

Note that Res$(I_{x_{F_{j,j+1,j+2}}})$ is minimally generated by
monomials
$$\frac{m_{F_{j,j+1,j+2}}}{gcd(m_{F_{i,i+1,i+2}},m_{F_{j,j+1,j+2}})}=x_{j}x_{j+1}x_{j+2},$$ where $i,j\in \{1,4,\ldots,3n-2\}$
and $i < j$.Therefore, the facet ideal
$I_\mathcal{F}(\Delta_{\Gamma'}(Y_{3,n}))$ does not have linear
residuals for any monomial ordering of minimal system of generators of
$I_\mathcal{F}(\Delta_{\Gamma'}(Y_{3,n}))$. Hence
$\Delta_{\Gamma'}(Y_{3,n})$ is not shellable.
\end{Example}
\begin{Example}
\rm{The anti-Gallai indices associated to friendship graph $G=F_{n}$ are given by
$$\Omega_{\Gamma'}(F_n)=\{F_{1,2,2n+1},F_{3,4,2n+1},\ldots,F_{2n-1,2n,2n+1}\}.$$
So, the anti-Gallai simplicial complex $\Delta_{\Gamma'}(F_n)$
consisting of $n$ facets with a common vertex is pure of dimension
$2$, as shown in Figure ~\ref{fig:f9}.

\begin{figure}[htbp]
        \centerline{\includegraphics[width=8.0cm]{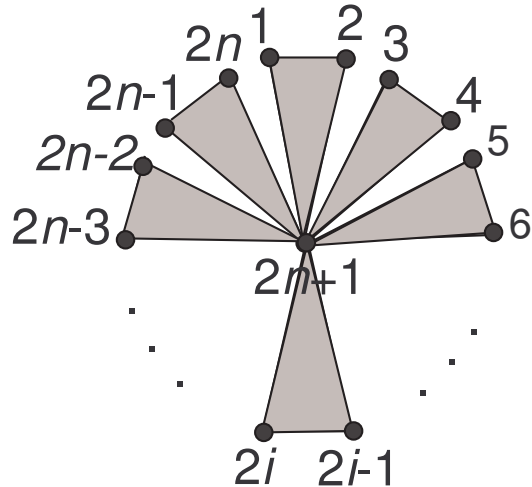}}
        \caption {Anti-Gallai simplicial complex $\Delta_{\Gamma'}(F_n)$}
         \label{fig:f9}
        \end{figure}}
The residual Res$(I_{x_{F_{j,j+1,2n+1}}})$ is minimally generated by monomials
$$\frac{m_{F_{j,j+1,2n+1}}}{gcd(m_{F_{i,i+1,2n+1}},m_{F_{j,j+1,2n+1}})}=x_{j}x_{j+1},$$ where $i,j\in \{1,3,\ldots,2n-1\}$
and $i<j$. Therefore, the facet ideal
$I_\mathcal{F}(\Delta_{\Gamma'}(F_n)$ does not have linear residuals
for any monomial ordering of minimal system of generators of
$I_\mathcal{F}(\Delta_{\Gamma'}(F_n)$. Hence $\Delta_{\Gamma'}(F_n)$
is not shellable.
\end{Example}

\end{document}